\documentclass{amsart}
\usepackage{amsrefs}
\usepackage{amsthm}
\usepackage{graphicx,color}
\usepackage[english]{babel}  
\usepackage[latin1]{inputenc}
\usepackage[mathscr]{eucal}
\usepackage{amssymb}
\usepackage[dvipsnames]{xcolor} 
\begin{document}

\title[Fractional Schr\"odinger-Poisson equations]
{Fractional Schr\"odinger-Poisson equations with general nonlinearities}

\author[Ronaldo C. Duarte]{Ronaldo C. Duarte}


\email{ronaldocesarduarte@gmail.com}
\email{marco@dme.ufcg.edu.br}



\author[Marco A. S. Souto]{Marco A. S. Souto}

\keywords{positive solutions, ground state solutions, periodic potential}

\subjclass{Primary 35J60; Secondary  35J10}

\begin{abstract}
In this paper we investigate the existence of positive solutions and ground state solutions for a class of fractional Schr\"odinger-Poisson equations
in $\mathbb R^3$  with general nonlinearity.
\end{abstract}

\maketitle

\newtheorem{theorem}{Theorem}[section]
\newtheorem{lemma}[theorem]{Lemma}
\newtheorem{proposition}[theorem]{Proposition}
\newtheorem{corollary}[theorem]{Corollary}
\newtheorem{remark}[theorem]{Remark}
\newtheorem{definition}[theorem]{Definition}
\renewcommand{\theequation}{\thesection.\arabic{equation}}

\section{Introduction}

In this article we consider the Schr\"odinger-Poisson system
$$
\left\{\begin{array}{lcl}
\left(-\Delta\right)^{s} u +V(x)u+\phi u= f(u),  &\mbox{ in }& \mathbb R^3, \\ \left(-\Delta\right)^{t} \phi=u^2, &\mbox{ in }& \mathbb R^3,
\end{array}\right. \leqno (P)
$$
where $\left(-\Delta\right)^\alpha$ is the fractional Laplacian for $\alpha = s,t$. This paper was motivated by \cite{amsss}. In \cite{amsss} the authors show the existence of positive solutions for the system 
$$
\left\{\begin{array}{lcl}
- \Delta u +V(x)u+\phi u= f(u),  &\mbox{ in }& \mathbb R^3, \\ -\Delta \phi=u^2, &\mbox{ in }& \mathbb R^3,
\end{array}\right.
$$
where $V:\mathbb{R}^3 \to \mathbb R$ is a continuous periodic potential and positive. Our purpose is to show that when we consider this system with fractional Laplacian operator instead of the Laplacian, then we get a positive solution and a ground state solution for the system. We emphasize that we prove the existence of weak solution to the system and without using results of regularity,  we show that the weak solution is positive almost everywhere in $\mathbb{R}^{3}$. To prove this, we present another version of the Logarithmic lemma and we deduce a weak comparison principle for the solution of the system (See Theorem \ref{lm41}).

We will admit that the potential $V$ satisfies,
\begin{enumerate}
\item[$(V_o)$\ ]  $V(x) \geq \alpha >0$, $\forall x \in \mathbb R^3$, for some constant $\alpha >0$,

\item[$(V_1)$\ ]  $V(x)=V(x+y)$,   for all $x \in \mathbb R^3$, $y\in \mathbb Z^3$.\newline
\end{enumerate}
Also, we will assume that
$f\in C(\mathbb R,\mathbb R)$ is a function satisfying:
\begin{enumerate}
\item[$(f_1)$\ ]
$f(u)u>0$, $u \neq 0$;

\item[$(f_2)$\ ]
$\displaystyle\lim_{u\rightarrow 0} \frac{f(u)}{u} = 0$;

\item[$(f_3)$\ ] there exists $p \in (4,2^{*}_{s})$ and $C>0$, such that
$$|f(u)|\leq C(|u|+|u|^{p-1}),$$ for all $u \in \mathbb{R}$, where $2^{*}_{s}=\frac{6}{3-2s}$.

\item[$(f_4)$\ ]
$\displaystyle\lim_{u\rightarrow +\infty} \frac{F(u)}{u^4} =+\infty$, where $F(u)=\int_0^u f(z)dz$;

\item[$(f_5)$\ ]The function
$u \longmapsto \frac{f(u)}{|u|^{3}}$ is increasing 
in $|u|\neq 0$.
\end{enumerate}
We will denote $g(u):=f(u^{+})$ and $G(t)=\int_{0}^{t}g(s)ds$.

The System $(P)$ was studied in \cite{bf}. The author studied the following one dimensional system
$$
\left\{\begin{array}{lcl}
- \Delta u +\phi u= a|u|^{p-1}u,  &\mbox{ in }& \mathbb R, \\ \left(-\Delta\right)^{t} \phi=u^2, &\mbox{ in }& \mathbb R,
\end{array}\right. 
$$
for $p\in(1,5)$ and $t \in (0,1)$. In \cite{zjs2}, the authors show the existence of positive solutions for the system
$$
\left\{\begin{array}{lcl}
- \Delta u + u + \lambda \phi u= f(u),  &\mbox{ in }& \mathbb{R}^{3}, \\ -\Delta \phi= \lambda u^2, &\mbox{ in }& \mathbb R,
\end{array}\right. 
$$
for $\lambda>0$ and general critical nonlinearity, $f$.
In \cite{zjs}, the authors have proved the existence of radial ground state solutions of $(P)$ when $V=0$. In \cite{zhang}, the system was studied, although the sign of the solutions is not considered. In this paper, we prove the existence of positive solutions for system $(P)$. Moreover, by the method of the Nehari manifold, we ensure the existence of a ground state solution for the problem.

Our result is:

\begin{theorem}{\label{fth}}
Suppose that $s \in (\frac{3}{4},1)$, $t \in (0,1)$, $V$ satisfies $(V_o)$ and $(V_1)$, and  $f$  satisfies $(f_1)- (f_5)$. Then the system ($P$) has a positive solution and a ground state solution.
\end{theorem}

The hypothesis  $s \in (\frac{3}{4},1)$ is required to ensure that the interval $(4,2^{\ast}_{s})$ is nondegenerate. 

\begin{remark}{\label{rm1}} The condition $(f_5)$ implies that
 $H(u)=uf(u)-4F(u)$ is a non-negative function.
\end{remark}

In the paper \cite{gsd}, Lemma 2.3, the authors proved another version of the Lions lemma. We will need this lemma to prove our result. It states that:
\begin{lemma}\label{l1.3}
If $\left\{u_{n}\right\}_{n \in \mathbb{N}}$ is a bounded sequence in $H^{s}(\mathbb{R}^{3})$ such that for some $R>0$ and $2\leq q< 2^{\ast}_{s}$ we have
$$
\sup_{x \in \mathbb{R}^{3}}\int_{B_{R}(x)}|u_{n}|^{q} \longrightarrow 0
$$
when $n \rightarrow \infty$, then $u_{n}\rightarrow 0$ in $L^{r}(\mathbb{R}^{3})$ for all $r \in (2,2^{\ast}_{s})$.
\end{lemma}

\section{Some preliminary results}

Let $s \in (0,1)$, we denote by $\dot{H}^{s}(\mathbb{R}^{3})$ the homogeneous fractional space. It is defined as the completion of $C_{0}^{\infty}(\mathbb{R}^{3})$ under the norm
$$
||u||_{\dot{H}^{s}}=\left(\int_{\mathbb{R}^{3}}\int_{\mathbb{R}^{3}}\frac{(u(x)-u(y))^{2}}{|x-y|^{3+2s}}dxdy\right)^{\frac{1}{2}}
$$
and we define 
$$
H^{s}(\mathbb{R}^{3}):=\left\{u \in L^{2}(\mathbb{R}^{3});\int_{\mathbb{R}^{3}}\int_{\mathbb{R}^{3}}\frac{(u(x)-u(y))^{2}}{|x-y|^{3+2s}}dxdy<\infty \right\}.
$$ 
The space $H^{s}(\mathbb{R}^{3})$ is a Hilbert space with the norm
$$
||u||_{H^{s}}=\left(\int_{\mathbb{R}^{3}}|u|^{2}dx+\int_{\mathbb{R}^{3}}\int_{\mathbb{R}^{3}}\frac{(u(x)-u(y))^{2}}{|x-y|^{3+2s}}dxdy\right)^{\frac{1}{2}}
$$
We define the fractional Laplacian operator
$\left(-\Delta\right)^{s}:\dot{H}^{s}(\mathbb{R}^{3})\longrightarrow (\dot{H}^{s}(\mathbb{R}^{3}))'$ by  $(\left(-\Delta\right)^{s}u,v)=\frac{\zeta}{2}(u,v)_{\dot{H}^{s}}$, where
$
\zeta=\zeta(s)=\left(\int_{\mathbb{R}^{3}}\frac{1-cos(\xi_{1})}{|\xi|^{3+2s}}d \xi \right)^{-1}
$ and $(\cdot,\cdot)_{H^{s}}$ is an inner product of $H^{s}(\mathbb{R}^{3})$. The constant $\zeta$ satisfies 
$$
\int_{\mathbb{R}^{3}}\int_{\mathbb{R}^{3}}\frac{(u(x)-u(y))(v(x)-v(y))}{|x-y|^{3+2s}}dxdy=2\zeta^{-1}\int_{\mathbb{R}^{3}}|\xi|^{2s}\mathcal{F}u(\xi) \overline{\mathcal{F}v(\xi)}d \xi,
$$
where $\mathcal{F}u$ is the Fourier transform of $u$ (see Proposition 3.4 of \cite{dpv}). The fractional Laplacian operator is a bounded linear operator.

A pair $(u, \phi_{u})$ is a solution of $(P)$ if
$$
\frac{\zeta(t)}{2}\int_{\mathbb{R}^{3}}\int_{\mathbb{R}^{3}}\frac{(\phi_{u}(x)-\phi_{u}(y))(w(x)-w(y))}{|x-y|^{3+2t}}dxdy = \int_{\mathbb{R}^{3}}u^{2}wdx.
$$
for all $w \in \dot{H}^{t}(\mathbb{R}^{3})$, and
$$
(\left(-\Delta\right)^{s}u,v)+\int_{\mathbb{R}^{3}}V(x)uvdx+\int_{\mathbb{R}^{3}}\phi_{u}uvdx=\int_{\mathbb{R}^{3}}f(u)vdx
$$
for all $v\in H^{s}(\mathbb{R}^{3})$.

Let us recall some facts about the Schr\"odinger-Poisson equations (see \cite{Ruiz,ap,zz,G} for instance). We can transform $(P)$ into a fractional Schr\"odinger problem with a nonlocal term. For all $u\in H^{s}(\mathbb R^3)$, there exists a unique $\phi=\phi_u \in \dot{H}^{t}(\mathbb R^3)$ such that
$$
\left(-\Delta\right)^{t} \phi=u^2.
$$
In fact, since $H^{s}(\mathbb R^3)\hookrightarrow L^{\frac{22^{\ast}_{t}}{2^{\ast}_{t}-1}}(\mathbb R^3)$ (continuously), a simple application of the Lax-Milgram theorem shows that $\phi_u$ is well defined and
$$
||\phi_{u}||_{\dot{H}^{t}}^{2}\leq S^2 ||u||^4_{\frac{22^{\ast}_{t}}{2^{\ast}_{t}-1}},
$$
where $||.||_p$ denotes the $L^p(\mathbb R^3)$ norm and $S$ is the best constant of the Sobolev immersion $H^s(\mathbb R^3) \rightarrow L^{2^{\ast}_{t}}(\mathbb R^3)$, that is
$$
S= \inf_{u \in \dot{H}^{t}(\mathbb{R}^{3})\setminus \left\{0\right\}}\frac{||u||_{\dot{H}^{t}}^{2}}{||u||_{2^{\ast}_{t}}^{2}}.
$$

\vspace{1mm}

\begin{lemma}{\label{lm1}}
We have:

 $i)$ there exists $C>0$ such that $||\phi_u||_{\dot{H}^{t}}\leq C||u||_{H^{s}}^2$ and
$$
\int_{\mathbb{R}^{3}}\int_{\mathbb{R}^{3}}\frac{(\phi_{u}(x)-\phi_{u}(y))^{2}}{|x-y|^{3+2t}}dxdy \leq C||u||_{{H}^{s}}^{4}
$$ for all $u\in H^s(\mathbb R^3)$;

$ii)$ $\phi_u\geq 0$, $\forall u\in H^s(\mathbb R^3)$;

$iii)$ $\phi_{tu}=t^2\phi_u$, $\forall t>0, u\in H^s(\mathbb R^3)$.

$iv)$ If $\tilde{u}(x):=u(x+z)$ then $\phi_{\tilde{u}}(x) = \phi_{u}(x+z)$ and
$$
\int_{\mathbb{R}^{3}}\phi_{u}u^{2}dx = \int_{\mathbb{R}^{3}}\phi_{\tilde{u}}\tilde{u}^{2}dx.
$$ for all $z \in \mathbb{R}^{3}$ and $u \in H^{s}(\mathbb{R}^{3})$.

$v)$ 
If $\left\{u_{n}\right\}$  converges weakly to $u$ in $H^{s}(\mathbb{R}^{3})$, then $\left\{\phi_{u_{n}}\right\}$  converges weakly to $\phi_{u}$ in $\dot{H}^{t}(\mathbb{R}^{3})$.
\end{lemma}
The proof is analogous to the case of Poisson equation in $\mathcal{D}^{1,2}(\mathbb{R}^{3})$ (See \cite{amsss, Ruiz, zz}). 

At first, we are interested in showing the existence of a positive solution for $(P)$. We will consider the following Euler-Lagrange functional
$$
\begin{array}{cccl}
I:&H^{s}(\mathbb{R}^{3})&\longrightarrow&\mathbb{R}\\
&u& \longmapsto &
\begin{array}{ll}\frac{\zeta(s)}{4}\int_{\mathbb{R}^{3}}\int_{\mathbb{R}^{3}}&\frac{(u(x)-u(y))^{2}}{|x-y|^{3+2s}}dxdy+\frac{1}{2}\int_{\mathbb{R}^{3}}V(x)u^{2}dx\\
&+\frac{1}{4}\int_{\mathbb{R}^{3}}\phi_{u}u^{2}dx-\int_{\mathbb{R}^{3}}G(u)dx,
\end{array}
\end{array}
$$
whose derivative is
$$
\begin{array}{ll}
I^{'}(u)(v)&=\frac{\zeta}{2}\int_{\mathbb{R}^{3}}\int_{\mathbb{R}^{3}}\frac{(u(x)-u(y))(v(x)-v(y))}{|x-y|^{3+2s}}dxdy\\
& +\int_{\mathbb{R}^{3}}V(x)uvdx
+\int_{\mathbb{R}^{3}}\phi_{u}u vdx-\int_{\mathbb{R}^{3}}g(u)vdx \\
&= (\left(-\Delta\right)^{s}u,v)+\int_{\mathbb{R}^{3}}V(x)uvdx
+\int_{\mathbb{R}^{3}}\phi_{u}u vdx-\int_{\mathbb{R}^{3}}g(u)vdx.
\end{array}
$$
Remark that critical points of $I$ determine solutions for $P$.

\begin{lemma}
The function 
$$
u \longmapsto ||u||:=\left(\frac{\zeta(s)}{2}\int_{\mathbb{R}^{3}}\int_{\mathbb{R}^{3}}\frac{(u(x)-u(y))^{2}}{|x-y|^{3+2s}}dxdy + \int_{\mathbb{R}^{3}}V(x)u^{2}dx\right)^{\frac{1}{2}}
$$
defines a norm in $H^{s}(\mathbb{R}^{3})$ wich is equivalent to the standard norm.
\end{lemma}

The proof of the previous lemma is trivial and therefore we will omit it in this paper. 
\section{Existence of the Solution}

\begin{theorem}\label{th31}
Suppose that $1>s> \frac{3}{4}$, $t \in (0,1)$, $V$ satisfies $(V_{0}), (V_{1})$ and $f$ satisfies $(f_{1})-(f_{4})$. Then, the problem $(P)$ has a nontrivial solution.
\end{theorem}

\begin{proof}
By usual arguments, we prove that the functional $I$ has the mountain pass geometry. By Montain Pass theorem, there is a Cerami's sequence for $I$ at the mountain pass level c. That is, there is $\left\{u_{n}\right\}_{n \in \mathbb{N}} \subset H^{s}(\mathbb{R}^{3})$ such that
$$
I(u_{n})\rightarrow c
$$
and
$$
(1+||u_{n}||)I^{'}(u_{n}) \rightarrow 0.
$$
where 
$$
c= \inf_{\gamma \in \Gamma}\sup_{t \in [0,1]}I(\gamma(t))
$$
and
$$
\Gamma=\left\{\gamma \in C([0,1], H^{s}(\mathbb{R}^{3}));\gamma(0)=0, \gamma(1)=e \right\},
$$
where $e \in H^{s}(\mathbb{R}^{3})$, and $e$ satisfies $I(e)<0$. By Remark \ref{rm1}
$$
\begin{array}{ll}
 4I(u_{n})-I'(u_{n})u_{n}
	&=||u_{n}||^{2}+\int_{\mathbb{R}^{3}}[f(u_{n})u_{n}-4F(u_{n})]dx \\
& \geq ||u_{n}||^{2}
\end{array}
$$
Therefore $\left\{u_{n}\right\}$ is bounded in $H^{s}(\mathbb{R}^{3})$. So, there is $u \in H^{s}(\mathbb{R}^{3})$ such that $\left\{u_{n}\right\}$ converges weakly to $u$. The Lemma \ref{lm1}, $(f_{2})$, and $(f_{3})$ imply that $u$ is a critical point for $I$. If $u \neq 0$ then $u$ is a nontrivial solution for 
$(P)$. Suppose that $u = 0$. 
We claim that $\{u_{n}\}$ does not converge to $0$ in $L^{r}(\mathbb{R}^{3})$ for all $r \in (2,2^{\ast}_{s})$. Indeed, otherwise, by $(f_{2})$, $(f_{3})$ and the boundedness of $\{u_{n}\}$ in  $L^{2}(\mathbb{R}^{3})$ 
we have
$$
\int_{\mathbb{R}^{3}}g(u_{n})u_{n}dx \rightarrow 0;
$$
By Lemma \ref{lm1}
$$
\begin{array}{ll}
||u_{n}||^{2}& \leq ||u_{n}||^{2}+\int_{\mathbb{R}^{3}}\phi_{u_{n}}u_{n}^{2}dx \\
& = \int_{\mathbb{R}^{3}}g(u_{n})u_{n}dx + I'(u_{n})u_{n}.
\end{array}
$$
The right side of this last inequality converges to $0$. In this case, $u_{n}\rightarrow 0$ in $H^{s}(\mathbb{R}^{3})$. Consequently
$$
c=\lim I(u_{n})=0.
$$
This last equality can not occur. Then, we can  assume that there are $R>0$ and $\delta>0$ such that passing to a subsequence if necessary
$$
\int_{B_{R}(y_{n})}u_{n}^{2}dx\geq \delta,
$$
for some sequence $\{y_{n}\} \subset \mathbb{Z}^{3}$ (See Lemma \ref{l1.3}).
For each $n \in \mathbb{N}$, we define
$$
w_{n}(x):=u_{n}(x+y_{n}).
$$
Note that $w_{n} \in H^{s}(\mathbb{R}^{3})$. Moreover, changing the variables in the integral below, we have
$$
\begin{array}{ll}
I(w_{n})
&= \begin{array}{ll}\frac{\zeta}{4}\int_{\mathbb{R}^{3}}\int_{\mathbb{R}^{3}}&\frac{(u_{n}(x+y_{n})-u_{n}(y+y_{n}))^{2}}{|(x+y_{n})-(y+y_{n})|^{3+2s}}dxdy+\frac{1}{2}\int_{\mathbb{R}^{3}}V(x)u_{n}(x+y_{n})^{2}dx\\
&+\frac{1}{4}\int_{\mathbb{R}^{3}}\phi_{w_{n}}w_{n}^{2}dx-\int_{\mathbb{R}^{3}}G(u_{n}(x+y_{n}))dx
\end{array} \\
& = \begin{array}{ll}\frac{\zeta}{4}\int_{\mathbb{R}^{3}}\int_{\mathbb{R}^{3}}&\frac{(u_{n}(z)-u_{n}(w))^{2}}{|z-w|^{3+2s}}dzdw+\frac{1}{2}\int_{\mathbb{R}^{3}}V(z)u_{n}(z)^{2}dz\\
&+\frac{1}{4}\int_{\mathbb{R}^{3}}\phi_{u_{n}}u_{n}^{2}dx-\int_{\mathbb{R}^{3}}G(u_{n}(z))dz. 
\end{array}\\
& = I(u_{n}).
\end{array}
$$
Analogously, for every $\phi \in H^{s}(\mathbb{R}^{3})$
$$
\begin{array}{ll}
I'(w_{n})\phi&=\begin{array}{ll}\frac{\zeta}{2}\int_{\mathbb{R}^{3}}\int_{\mathbb{R}^{3}}&\frac{(w_{n}(x)-w_{n}(y))(\phi(x)-\phi(y))}{|x-y|^{3+2s}}dxdy+\int_{\mathbb{R}^{3}}V(x)w_{n}\phi dx\\
&+\int_{\mathbb{R}^{3}}\phi_{w_{n}}w_{n}\phi dx-\int_{\mathbb{R}^{3}}g(w_{n})\phi dx 
\end{array} \\
& = \begin{array}{ll}\frac{\zeta}{2}\int_{\mathbb{R}^{3}}\int_{\mathbb{R}^{3}}&\frac{(u_{n}(x+y_{n})-u_{n}(y+y_{n}))(\phi(x)-\phi(y))}{|(x+y_{n})-(y+y_{n})|^{3+2s}}dxdy\\&+\int_{\mathbb{R}^{3}}V(x+y_{n})u_{n}(x+y_{n})\phi(x) dx\\
&+\int_{\mathbb{R}^{3}}\phi_{u_{n}}(x+y_{n})u_{n}(x+y_{n})\phi dx\\
&-\int_{\mathbb{R}^{3}}g(u_{n}(x+y_{n}))\phi(x) dx 
\end{array} \\

& = \begin{array}{ll}\frac{\zeta}{2}\int_{\mathbb{R}^{3}}\int_{\mathbb{R}^{3}}&\frac{(u_{n}(z)-u_{n}(w)(\phi(z-y_{n})-\phi(w-y_{n}))}{|z-w|^{3+2s}}dzdw\\
&+\int_{\mathbb{R}^{3}}V(z)u_{n}(z)\phi(z-y_{n}) dz\\
&+\int_{\mathbb{R}^{3}}\phi_{u_{n}}(z)u_{n}(z)\phi(z-y_{n}) dz\\
&-\int_{\mathbb{R}^{3}}g(u_{n}(z))\phi(z-y_{n}) dz 
\end{array} \\

& = I'(u_{n})\overline{\phi}
\end{array}
$$
where $\overline{\phi}(x)=\phi(x-y_{n})$. This implies that $\{w_{n}\}$ is a Cerami's sequence for $I$ at the level $c$. Analogously, we can show that $\{w_{n}\}$ is bounded, $\{w_{n}\}$ converges weakly to some $w_{0}\in H^{s}(\mathbb{R}^{3})$ and that $I'(w_{0})=0$. Passing to a subsequence, if necessary, we can assume that $\{w_{n}\}$ converges on $L^{2}_{loc}(\mathbb{R}^{3})$ to $w_{0}$. Then
$$
\begin{array}{ll}
\int_{B_{R}(0)}w_{0}^{2}dx & = \lim\limits_{n \rightarrow \infty}\int_{B_{R}(0)}w_{n}^{2}dx \\
& = \lim\limits_{n \rightarrow \infty}\int_{B_{R}(0)}u_{n}(x+y_{n})^{2}dx \\
& =  \lim\limits_{n \rightarrow \infty}\int_{B_{R}(y_{n})}u_{n}(z)^{2}dz \geq \delta.
\end{array}
$$
Therefore, $w_{0}$ is a nontrivial solution for $(P)$. Thus, if $u=0$ we prove that there is a critical point for $I$, that is nontrivial.
\end{proof}

\section{Positivity of the solution}

In this section, we will prove that the solution of Theorem \ref{th31} is positive. Initially, we will prove a version of Logarithmic lemma. The Logarithmic lemma was presented by Di Castro, Kuusi and Palatucci. (lemma 1.3 of \cite{dkp}). In the Logarithmic lemma, the authors give an estimate for weak solutions of the equation
$$
\left\{\begin{array}{rcccc}
\left(-\Delta_{p}\right)^{s}u&=&0&in& \Omega \\
u&=&g&in& \mathbb{R}^{n}\setminus\Omega
\end{array}
\right.
$$
in $B_{r}(x_{0})\subset B_{\frac{R}{2}}(x_{0}) \subset \Omega$, for $x_{0}\in \Omega$ and $u \geq 0$ in $B_{R}(x_{0})$. Following the ideas from Di Castro, Kuusi and Palatucci, we will show a similar estimate for a supersolution of the problem
$$
\begin{array}{lll}
\left(-\Delta\right)^{s}u+a(x)u=0&in &\mathbb{R}^{n}
\end{array}
$$
(See Lemma \ref{lm41} bellow). Supersolutions are defined in the following way
$$
\int_{\mathbb{R}^{n}}\int_{\mathbb{R}^{n}}\frac{(u(x)-u(y))\left(v(x)-v(y)\right)}{|x-y|^{n+2s}}dxdy+ \int_{\mathbb{R}^{n}}a(x)u(x)v(x)dx\geq 0,
$$
for all $v \in H^{s}(\mathbb{R}^{n})$ with $v \geq 0$ almost everywere. Also, in this situation, we need not to assume that $u\geq0$ in some subset of $\mathbb{R}^{n}$.
With this estimate, we conclude that the supersolution satisfies $u > 0$ almost everywere in $\mathbb{R}^{3}$ or $u=0$ almost everywere in $\mathbb{R}^{3}$.

\begin{lemma}\label{lm41}
Suppose that $a:\mathbb{R}^{n}\rightarrow \mathbb{R}$ is a nonnegative function and $u \in H^{s}(\mathbb{R}^{n})$. If 
$$
\int_{\mathbb{R}^{n}}\int_{\mathbb{R}^{n}}\frac{(u(x)-u(y))\left(v(x)-v(y)\right)}{|x-y|^{n+2s}}dxdy+ \int_{\mathbb{R}^{n}}a(x)u(x)v(x)dx\geq 0.
$$
for all $v \in H^{s}(\mathbb{R}^{n})$ with $v \geq 0$ almost everywere, then $u \geq 0$ almost everywere. In other words, if $\left(-\Delta\right)^{s}u+a(x)u\geq0$ then $u \geq 0$ almost everywere.
\end{lemma}

\begin{proof}
Define $v=u^{-} = \max\{0,-u\}$. By hypothesis
$$
\int_{\mathbb{R}^{n}}\int_{\mathbb{R}^{n}}\frac{(u(x)-u(y))\left(u^{-}(x)-u^{-}(y)\right)}{|x-y|^{n+2s}}dxdy+ \int_{\mathbb{R}^{n}}a(x)u(x)u^{-}(x)dx\geq 0.
$$
But,
\begin{itemize}
\item if $u(x)>0$ and $u(y)>0$ then $(u(x)-u(y))\left(u^{-}(x)-u^{-}(y)\right) =0$.

\item if $u(x)<0$ and $u(y)<0$ then $(u(x)-u(y))\left(u^{-}(x)-u^{-}(y)\right) = - (u(x)-u(y))^{2} \leq 0$.

\item if $u(x)>0$ and $u(y)<0$ then $(u(x)-u(y))\left(u^{-}(x)-u^{-}(y)\right) = (u(x)-u(y))u(y) \leq 0$.

\item if $u(x)<0$ and $u(y)>0$ then $(u(x)-u(y))\left(u^{-}(x)-u^{-}(y)\right) = (u(x)-u(y))(-u(x)) \leq 0$. 

\item if $u(x) < 0$, then $a(x)u(x)u^{-}(x) = -a(x)u^{2}(x)<0$, and $a(x)u(x)u^{-}(x)$ $= 0$ in the case $u(x)\geq0$.
\end{itemize}
We conclude that each one of the integrals above is equal to zero. Therewith

$$
\frac{(u(x)-u(y))\left(u^{-}(x)-u^{-}(y)\right)}{|x-y|^{n+2s}}= 0.
$$
Therefore $u^{-}$ is constant in $H^{s}(\mathbb{R}^{n})$, that is, $u^{-}=0$.
\end{proof}
\begin{lemma}\label{lm42}
Suppose that $\epsilon \in \left(\left.0,1\right]\right.$ and $a,b \in \mathbb{R}^{n}$. Then
$$
|a|^{2}\leq |b|^{2}+2\epsilon|b|^{2}+\frac{1+\epsilon}{\epsilon}|a-b|^{2}
$$
\end{lemma}
\begin{proof}
$$
\begin{array}{ll}
|a|^{2} & \leq \left(|b|+|a-b|\right)^{2} \\
& = |b|^{2}+2|b||a-b|+|a-b|^{2} \\
\end{array}
$$
By Cauchy inequality with $\epsilon$ 
$$
|b||a-b|\leq \epsilon|b|^{2}+\frac{|a-b|^{2}}{4\epsilon} \leq 
\epsilon|b|^{2}+\frac{|a-b|^{2}}{2\epsilon} 
$$
Replacing in the inequality above
$$
\begin{array}{ll}
|a|^{2} & \leq |b|^{2}+2\epsilon|b|^{2}+\frac{|a-b|^{2}}{\epsilon}+|a-b|^{2} \\
& = |b|^{2}+2\epsilon|b|^{2}+\frac{1+\epsilon}{\epsilon}|a-b|^{2}.
\end{array} 
$$
\end{proof}
\begin{lemma}\label{lm43}
With the same assumptions of Lemma \ref{lm41} and $a \in L^{1}_{loc}(\mathbb{R}^{3})$, we have for all $r,d>0$ and $x_{0} \in \mathbb{R}^{n}$  
\begin{equation}
\int_{B_{r}}\int_{B_{r}}\left|\log\left(\frac{d+u(x)}{d+u(y)}\right)\right|^{2}\frac{1}{|x-y|^{n+2s}}dxdy 
\leq Cr^{n-2s} + \int_{B_{2r}}a(x)dx, 
\end{equation}
where $B_{r}=B_{r}(x_{0})$ and $C=C(n,s)>0$ is a constant.
\end{lemma}

\begin{proof}
Consider $\phi \in C_{0}^{\infty}(B_{\frac{3r}{2}})$, $0\leq \phi \leq 1$, $\phi = 1$ in $B_{r}$ and $K>0$ such that $||D\phi||_{\infty} \leq Kr^{-1}$. The function
$$
\eta=\frac{\phi^{2}}{u+d}
$$
is in $H^{s}(\mathbb{R}^{n})$ and $\eta\geq0$ (see Lemma 5.3 in \cite{dpv}). By hypothesis
$$
\begin{array}{ll}
0&\leq \int_{R^{n}}\int_{R^{n}}\frac{(u(x)-u(y))(\eta(x)-\eta(y))}{|x-y|^{n+2s}}dxdy + \int_{\mathbb{R}^{n}}a(x)u(x)\eta(x)dx \\
&= \int_{B_{2r}}\int_{B_{2r}}\frac{(u(x)-u(y))(\eta(x)-\eta(y))}{|x-y|^{n+2s}}dxdy \\
& + \int_{R^{n}-B_{2r}}\int_{B_{2r}}\frac{(u(x)-u(y))(\eta(x)-\eta(y))}{|x-y|^{n+2s}}dxdy \\
& +
\int_{B_{2r}}\int_{R^{n}-B_{2r}}\frac{(u(x)-u(y))(\eta(x)-\eta(y))}{|x-y|^{n+2s}}dxdy \\
& +
\int_{R^{n}-B_{2r}}\int_{R^{n}-B_{2r}}\frac{(u(x)-u(y))(\eta(x)-\eta(y))}{|x-y|^{n+2s}}dxdy \\

& + \int_{\mathbb{R}^{n}}a(x)u(x)\eta(x)dx.
\end{array}
$$
We will prove some statements about the five integrals  of the last inequality.

\begin{itemize}
 \item $A.1)$ There are constants $C_{2},C_{3}>0$, such that, they depend only on $n$ and $s$ and
$$
\begin{array}{ll}
&\int_{B_{2r}}\int_{B_{2r}}\frac{(u(x)-u(y))(\eta(x)-\eta(y))}{|x-y|^{n+2s}}dxdy \\
& \leq -C_{2}\int_{B_{2r}}\int_{B_{2r}}\left|\log\left(\frac{d+u(x)}{d+u(y)}\right)\right|^{2}\frac{1}{|x-y|^{n+2s}}\min\left\{\phi(y)^{2}, \phi(x)^{2}\right\}dxdy \\
&+C_{3}\int_{B_{2r}}\int_{B_{2r}}\frac{|\phi(x)-\phi(y)|^{2}}{|x-y|^{n+2s}}dxdy,
\end{array}
$$
where $\min\left\{a,b\right\} = a$ if $a \leq b$ and $\min\left\{a,b\right\} = b$ if $a \geq b$, for all $a,b \in \mathbb{R}$.
\end{itemize}

Fix $x,y \in B_{2r}$ and suppose that $u(x)>u(y)$. Define
$$
\epsilon = \delta\frac{u(x)-u(y)}{u(x)+d}
$$
where $\delta \in (0,1)$ is chosen small enough such that $\epsilon \in (0,1)$. Taking $a= \phi(x)$ and $b=\phi(y)$ in the Lemma \ref{lm42}, we get
$$
|\phi(x)|^{2}\leq |\phi(y)|^{2}+2 \delta\frac{u(x)-u(y)}{u(x)+d}|\phi(y)|^{2}+\left(\delta^{-1}\frac{u(x)+d}{u(x)-u(y)}+1\right)|\phi(x)-\phi(y)|^{2}
$$
Replacing 
$$
\begin{array}{ll}
&\frac{(u(x)-u(y))(\eta(x)-\eta(y))}{|x-y|^{n+2s}}\\
&=(u(x)-u(y))\left(\frac{\phi^{2}(x)}{u(x)+d}-\frac{\phi^{2}(y)}{u(y)+d}\right)\frac{1}{|x-y|^{n+2s}}\\
&  
\begin{array}{ll}
&\leq (u(x)-u(y))\left(\frac{ |\phi(y)|^{2}+2 \delta\frac{u(x)-u(y)}{u(x)+d}|\phi(y)|^{2}+\left(\delta^{-1}\frac{u(x)+d}{u(x)-u(y)}+1\right)|\phi(x)-\phi(y)|^{2}}{u(x)+d}\right.\\&
\left.-\frac{\phi^{2}(y)}{u(y)+d}\right)\frac{1}{|x-y|^{n+2s}} 
\end{array}
\\ &  
\begin{array}{ll}
&=(u(x)-u(y))\frac{|\phi(y)|^{2}}{u(x)+d}\left[ 1+ 2\delta\frac{u(x)-u(y)}{u(x)+d}+\left(\delta^{-1}\frac{u(x)+d}{u(x)-u(y)}+1\right)\frac{|\phi(x)-\phi(y)|^{2}}{|\phi(y)|^{2}}\right.\\
&\left.-\frac{u(x)+d}{u(y)+d}\right]\frac{1}{|x-y|^{n+2s}}
\end{array}
 \\
& 
\begin{array}{ll} &=(u(x)-u(y))\frac{|\phi(y)|^{2}}{u(x)+d}\frac{1}{|x-y|^{n+2s}}\left( 1+ 2\delta\frac{u(x)-u(y)}{u(x)+d}-\frac{u(x)+d}{u(y)+d}\right)\\
& + \left(\delta^{-1} + \frac{(u(x)-u(y))}{u(x)+d}\right)|\phi(x)-\phi(y)|^{2}\frac{1}{|x-y|^{n+2s}}
\end{array}
\\
&
\begin{array}{ll}
 &\leq (u(x)-u(y))\frac{|\phi(y)|^{2}}{u(x)+d}\frac{1}{|x-y|^{n+2s}}\left( 1+ 2\delta\frac{u(x)-u(y)}{u(x)+d}-\frac{u(x)+d}{u(y)+d}\right) \\
 &+ 2\delta^{-1}|\phi(x)-\phi(y)|^{2}\frac{1}{|x-y|^{n+2s}}.
 \end{array}
 \\
\end{array}
$$
We will rewrite the first part of the sum appearing on the right side of the last inequality 
$$
\begin{array}{ll}
&(u(x)-u(y))\frac{|\phi(y)|^{2}}{u(x)+d}\frac{1}{|x-y|^{n+2s}}\left( 1+ 2\delta\frac{u(x)-u(y)}{u(x)+d}-\frac{u(x)+d}{u(y)+d}\right)\\
& = \left(\frac{u(x)-u(y)}{u(x)+d}\right)^{2}\phi(y)^{2}\frac{1}{|x-y|^{n+2s}}\left[ \frac{u(x)+d}{u(x)-u(y)}+ 2\delta-\frac{u(x)+d}{u(y)+d}\cdot\frac{u(x)+d}{u(x)-u(y)} \right]\\
& = \left(\frac{u(x)-u(y)}{u(x)+d}\right)^{2}\phi(y)^{2}\frac{1}{|x-y|^{n+2s}}\left[\frac{1-\frac{u(x)+d}{u(y)+d}}{1-\frac{u(y)+d}{u(x)+d}} + 2\delta \right].\\
\end{array}
$$
Define the function $g:(0,1) \rightarrow \mathbb{R}$ by
$$
g(t)= \frac{1-t^{-1}}{1-t}.
$$
It satisfies $g(t) \leq -\frac{1}{4}\frac{t^{-1}}{1-t}$ if $t \in \left(\left.0,\frac{1}{2}\right]\right.$ and $g(t)\leq -1$ for all $t \in (0,1)$. We have two cases. If $\frac{u(y)+d}{u(x)+d} \leq \frac{1}{2}$ then, we conclude that
$$
\begin{array}{ll}
&\left(\frac{u(x)-u(y)}{u(x)+d}\right)^{2}\phi(y)^{2}\frac{1}{|x-y|^{n+2s}}\left[\frac{1-\frac{u(x)+d}{u(y)+d}}{1-\frac{u(y)+d}{u(x)+d}} + 2\delta \right]\\
& \leq \left(\frac{u(x)-u(y)}{u(x)+d}\right)^{2}\phi(y)^{2}\frac{1}{|x-y|^{n+2s}}\left[-\frac{1}{4}\frac{\frac{u(x)+d}{u(y)+d}}{\frac{u(x)-u(y)}{u(x)+d}} + 2\delta \right]\\
& = \frac{u(x)-u(y)}{u(x)+d}\phi(y)^{2}\frac{1}{|x-y|^{n+2s}}\left[-\frac{1}{4}\frac{u(x)+d}{u(y)+d} + 2\delta\frac{u(x)-u(y)}{u(x)+d} \right] \\
& = \frac{u(x)-u(y)}{u(y)+d}\phi(y)^{2}\frac{1}{|x-y|^{n+2s}}\left[-\frac{1}{4} + 2\delta\frac{(u(x)-u(y))(u(y)+d)}{(u(x)+d)^{2}} \right] \\
&\leq \frac{u(x)-u(y)}{u(y)+d}\phi(y)^{2}\frac{1}{|x-y|^{n+2s}}\left[-\frac{1}{4} + 2\delta \right].
\end{array}
$$
In the last inequality, we use that 
$$
\frac{(u(x)-u(y))(u(y)+d)}{(u(x)+d)^{2}} \leq 1.
$$
Choosing $\delta=\frac{1}{16}$
we have
$$
\begin{array}{ll}
&\left(\frac{u(x)-u(y)}{u(x)+d}\right)^{2}\phi(y)^{2}\frac{1}{|x-y|^{n+2s}}\left[\frac{1-\frac{u(x)+d}{u(y)+d}}{1-\frac{u(y)+d}{u(x)+d}} + 2\delta \right]\\
& \leq -\frac{1}{8}\frac{u(x)-u(y)}{u(y)+d}\phi(y)^{2}\frac{1}{|x-y|^{n+2s}}\\
& \leq -\frac{1}{8}\left[\log\left(\frac{u(x)+d}{u(y)+d}\right)\right]^{2}\phi(y)^{2}\frac{1}{|x-y|^{n+2s}}.
\end{array}
$$
Above, we have used that
$
(\log(t))^{2}\leq t-1
$
for all $t\geq2$, and that $\frac{u(x)+d}{u(y)+d}\geq 2$.
But, if $ \frac{u(y)+d}{u(x)+d} > \frac{1}{2}$, then using that $g(t) \leq -1$ and that $\delta= \frac{1}{16}$ 
$$
\begin{array}{ll}
&\left(\frac{u(x)-u(y)}{u(x)+d}\right)^{2}\phi(y)^{2}\frac{1}{|x-y|^{n+2s}}\left[\frac{1-\frac{u(x)+d}{u(y)+d}}{1-\frac{u(y)+d}{u(x)+d}} + 2\delta \right]\\ &\leq \left(\frac{u(x)-u(y)}{u(x)+d}\right)^{2}\phi(y)^{2}\frac{1}{|x-y|^{n+2s}}\left[-1 + 2\delta \right]\\
& \leq-\frac{7}{8}\left(\frac{u(x)-u(y)}{u(x)+d}\right)^{2}\phi(y)^{2}\frac{1}{|x-y|^{n+2s}} \\

&\leq-\frac{7}{32}\left[\log\left(\frac{u(x)+d}{u(y)+d}\right)\right]^{2}\phi(y)^{2}\frac{1}{|x-y|^{n+2s}}.
\end{array}
$$
Here, we have used that
$$
\begin{array}{ll}
\left[\log\left(\frac{u(x)+d}{u(y)+d}\right)\right]^{2} & = \left[\log\left(1+ \frac{u(x)-u(y)}{u(y)+d}\right)\right]^{2} \\
& \leq 4\left(\frac{u(x)-u(y)}{u(x)+d}\right)^{2}.
\end{array}
$$
This is a consequence of
$$
\log(1+t)\leq t
$$
for all $t>0$, and that
$$
t=\frac{u(x)-u(y)}{u(y)+d}=\frac{u(x)-u(y)}{u(x)+d}\cdot \frac{u(x)+d}{u(y)+d}\leq 2\frac{u(x)-u(y)}{u(x)+d}.
$$
In short
$$
\begin{array}{ll}
 &(u(x)-u(y))\frac{|\phi(y)|^{2}}{u(x)+d}\frac{1}{|x-y|^{n+2s}}\left( 1+ 2\delta\frac{u(x)-u(y)}{u(x)+d}-\frac{u(x)+d}{u(y)+d}\right)\\
 &\leq -\frac{1}{8}\left[\log\left(\frac{u(x)+d}{u(y)+d}\right)\right]^{2}\phi(y)^{2}\frac{1}{|x-y|^{n+2s}}.
 \end{array}
$$
We have proved that, if $u(x)>u(y)$ then
$$
\begin{array}{ll}
&\frac{(u(x)-u(y))(\eta(x)-\eta(y))}{|x-y|^{n+2s}}\\
&\leq -\frac{1}{8}\left[\log\left(\frac{u(x)+d}{u(y)+d}\right)\right]^{2}\phi(y)^{2}\frac{1}{|x-y|^{n+2s}} + 32|\phi(x)-\phi(y)|^{2}\frac{1}{|x-y|^{n+2s}}.
\end{array}
$$
Integrating in $B_{2r}$ the last inequality, we get
$$
\begin{array}{ll}
&\int_{B_{2r}}\int_{B_{2r}}\frac{(u(x)-u(y))(\eta(x)-\eta(y))}{|x-y|^{n+2s}}dxdy\\
&=\int_{B_{2r}}\int_{\left\{x;u(x)>u(y)\right\}}\frac{(u(x)-u(y))(\eta(x)-\eta(y))}{|x-y|^{n+2s}}dxdy \\
&+\int_{B_{2r}}\int_{\left\{x;u(x)<u(y)\right\}}\frac{(u(x)-u(y))(\eta(x)-\eta(y))}{|x-y|^{n+2s}}dxdy \\
& \leq -\frac{1}{8}\int_{B_{2r}}\int_{\left\{x;u(x)>u(y)\right\}}\left[\log\left(\frac{u(x)+d}{u(y)+d}\right)\right]^{2}\phi(y)^{2}\frac{1}{|x-y|^{n+2s}}dxdy \\

& -\frac{1}{8}\int_{B_{2r}}\int_{\left\{x;u(x)<u(y)\right\}}\left[\log\left(\frac{u(y)+d}{u(x)+d}\right)\right]^{2}\phi(x)^{2}\frac{1}{|x-y|^{n+2s}}dxdy \\
& + 32\int_{B_{2r}}\int_{B_{2r}}|\phi(x)-\phi(y)|^{2}\frac{1}{|x-y|^{n+2s}}dxdy. \\
\end{array}
$$
But, using that $\left|\log(x)\right| = \left|\log \left(\frac{1}{x}\right)\right|$ for all $x\neq 0$, we obtain that
$$
\left[\log\left(\frac{u(y)+d}{u(x)+d}\right)\right]^{2} = \left[\log\left(\frac{u(x)+d}{u(y)+d}\right)\right]^{2}.
$$
Replacing
$$
\begin{array}{ll}
&\int_{B_{2r}}\int_{B_{2r}}\frac{(u(x)-u(y))(\eta(x)-\eta(y))}{|x-y|^{n+2s}}dxdy\\
& \leq -\frac{1}{8}\int_{B_{2r}}\int_{\left\{x; u(x)>u(y)\right\}}\left[\log\left(\frac{u(x)+d}{u(y)+d}\right)\right]^{2}\phi(y)^{2}\frac{1}{|x-y|^{n+2s}}dxdy \\
& - \frac{1}{8}\int_{B_{2r}}\int_{\left\{x;u(x)<u(y)\right\}}\left[\log\left(\frac{u(x)+d}{u(y)+d}\right)\right]^{2}\phi(x)^{2}\frac{1}{|x-y|^{n+2s}}dxdy\\
& + 32\int_{B_{2r}}\int_{B_{2r}}|\phi(x)-\phi(y)|^{2}\frac{1}{|x-y|^{n+2s}}dxdy \\
& \leq -\frac{1}{8}\int_{B_{2r}}\int_{\left\{x; u(x)>u(y)\right\}}\left[\log\left(\frac{u(x)+d}{u(y)+d}\right)\right]^{2}\min{\left\{\phi(y)^{2},\phi(x)^{2}\right\}}\frac{1}{|x-y|^{n+2s}}dxdy \\
& - \frac{1}{8}\int_{B_{2r}}\int_{\left\{x;u(x)<u(y)\right\}}\left[\log\left(\frac{u(x)+d}{u(y)+d}\right)\right]^{2}\min{\left\{\phi(y)^{2},\phi(x)^{2}\right\}}\frac{1}{|x-y|^{n+2s}}dxdy\\
& + 32\int_{B_{2r}}\int_{B_{2r}}|\phi(x)-\phi(y)|^{2}\frac{1}{|x-y|^{n+2s}}dxdy \\
& =  -\frac{1}{8}\int_{B_{2r}}\int_{B_{2r}}\left[\log\left(\frac{u(x)+d}{u(y)+d}\right)\right]^{2}\min{\left\{\phi(y)^{2},\phi(x)^{2}\right\}}\frac{1}{|x-y|^{n+2s}}dxdy\\
&+32\int_{B_{2r}}\int_{B_{2r}}|\phi(x)-\phi(y)|^{2}\frac{1}{|x-y|^{n+2s}}dxdy,
\end{array}
$$
Thus, we have proved the claim 1.
\begin{itemize}

\item $A.2)$ There is $C_{3}>0$ such that, it depends only on $s$ and $n$ and
$$
\int_{R^{n}-B_{2r}}\int_{B_{2r}}\frac{(u(x)-u(y))(\eta(x)-\eta(y))}{|x-y|^{n+2s}}dxdy \leq C_{3}r^{n-2s}.
$$
\end{itemize}
Indeed,

$$
\begin{array}{ll}
&\int_{R^{n}-B_{2r}}\int_{B_{2r}}\frac{(u(x)-u(y))(\eta(x)-\eta(y))}{|x-y|^{n+2s}}dxdy\\
&=\int_{\mathbb{R}^{n}-B_{2r}}\int_{\mathbb{R}^{n}}(u(x)-u(y))\left(\frac{\phi^{2}(x)}{u(x)+d}-\frac{\phi^{2}(y)}{u(y)+d}\right)\frac{1}{|x-y|^{n+2s}}dxdy \\
& =\int_{\mathbb{R}^{n}-B_{2r}}\int_{\mathbb{R}^{n}}|\phi(x)|^{2}\frac{u(x)-u(y)}{u(x)+d}\frac{1}{|x-y|^{n+2s}}dxdy \\
& \leq \int_{\mathbb{R}^{n}-B_{2r}}\int_{\mathbb{R}^{n}}|\phi(x)|^{2}\frac{1}{|x-y|^{n+2s}}dxdy
\end{array}
$$
In the last equality, we have used that $u(y)\geq 0$ and therefore
$$
\frac{u(x)-u(y)}{u(x)+d} \leq 1.
$$
A simple calculation shows that
$$
\begin{array}{ll}
&\int_{\mathbb{R}^{n}-B_{2r}}\int_{\mathbb{R}^{n}}|\phi(x)|^{2}\frac{1}{|x-y|^{n+2s}}dxdy 
\leq C_{3}r^{n-2s}\\
\end{array}
$$
and $C_{3}$ depends only on $n$ and $s$. Therefore we get the assertion 2.

\begin{itemize}
\item $A.3)$ We claim that $$\int_{\mathbb{R}^{n}}a(x)u(x)\eta(x)dx \leq \int_{B_{2r}}a(x)dx$$.
\end{itemize}
Indeed,
$$
\begin{array}{ll}
\int_{\mathbb{R}^{n}}a(x)u(x)\eta(x)dx& = \int_{\mathbb{R}^{n}}a(x)u(x)\frac{\phi^{2}(x)}{u(x)+d}dx \\
& =\int_{B_{2r}}a(x)u(x)\frac{\phi^{2}(x)}{u(x)+d}dx \\
& =\int_{B_{2r}}a(x)\frac{u(x)}{u(x)+d} \phi^{2}(x)dx\\
& \leq \int_{B_{2r}}a(x)dx
\end{array}
$$
We have used that $supp(\eta)\subset B_{2r}$, that $\phi(x)\in (0,1)$ and that $\frac{u(x)}{u(x)+d} \leq 1$.

The statements 1,2 and 3 imply that
$$
\begin{array}{ll}
 &\int_{B_{2r}}\int_{B_{2r}}\left[\log\left(\frac{u(x)+d}{u(y)+d}\right)\right]^{2}\min\left\{\phi(y)^{2}, \phi(x)^{2}\right\}\frac{1}{|x-y|^{n+2s}}dxdy \\
&\leq C_{5}\int_{B_{2r}}\int_{B_{2r}}\frac{|\phi(x)-\phi(y)|^{2}}{|x-y|^{n+2s}}dxdy +C_{6}r^{n-2s}+\int_{B_{2r}}a(x)dx.
\end{array}
$$
for constants $C_{5},C_{6}$. The constants $C_{5},C_{6}$ depend only on $n$ and $s$. But $\phi=1$ in $B_{r}$ implies that
\begin{equation}\label{eq1}
\begin{array}{ll}
&\int_{B_{r}}\int_{B_{r}}\left|\log\left(\frac{d+u(x)}{d+u(y)}\right)\right|^{2}\frac{1}{|x-y|^{n+2s}}dxdy \\
&\leq C_{5}\int_{B_{2r}}\int_{B_{2r}}\frac{|\phi(x)-\phi(y)|^{2}}{|x-y|^{n+2s}}dxdy +C_{6}r^{n-2s}+\int_{B_{2r}}a(x)dx
\end{array}
\end{equation}
Finally, we will show that 
$$
\int_{B_{2r}}\int_{B_{2r}}\frac{|\phi(x)-\phi(y)|^{2}}{|x-y|^{n+2s}}dxdy \leq C_{7}r^{n-2s}
$$
By hypothesis
$$
\begin{array}{ll}
\int_{B_{2r}}\int_{B_{2r}}\frac{|\phi(x)-\phi(y)|^{2}}{|x-y|^{n+2s}}dxdy &\leq Kr^{-2}\int_{B_{2r}}\int_{B_{2r}}\frac{|x-y|^{2}}{|x-y|^{n+2s}}dxdy \\
&=Kr^{-2}\int_{B_{2r}}\int_{B_{2r}}\frac{1}{|x-y|^{n+2(s-1)}}dxdy \\
&\leq Kr^{-2}\frac{r^{2(1-s)}}{2(1-s)}|B_{2r}|=C_{7}r^{n-2s}
\end{array}
$$
where $C_{7}$ depends only on $n$ and $s$. Replacing this last estimate in (\ref{eq1}) we obtain the Lemma \ref{lm43}.
\end{proof}
Following the same ideas of Theorem A.1 in \cite{bf2}, we will prove the theorem stated at the beginning of the section.
\begin{theorem}\label{th44}
Suppose that $u \in H^{s}(\mathbb{R}^{n})$ and $a\geq 0$ with $a \in L^{1}_{loc}(\mathbb{R}^{n})$. We will assume that
$$
\int_{\mathbb{R}^{n}}\int_{\mathbb{R}^{n}}\frac{(u(x)-u(y))\left(v(x)-v(y)\right)}{|x-y|^{n+2s}}dxdy+ \int_{\mathbb{R}^{n}}a(x)u(x)v(x)dx\geq 0,
$$
for all $v \in H^{s}(\mathbb{R}^{n})$ with $v \geq 0$ almost everywere.
Then $u > 0$ almost everywere in $\mathbb{R}^{n}$ or $u=0$ almost everywere in $\mathbb{R}^{n}$.
\end{theorem}

\begin{proof}
By Lemma \ref{lm41}, $u \geq 0$. Suppose that $x_{0}\in \mathbb{R}^{n}$ and $r>0$. Define
$$
Z:=\{x \in B_{r}(x_{0}); u(x)=0\}
$$
If $|Z|>0$, then we define
$$
\begin{array}{cccl}
F_{\delta}:&B_{r}(x_{0})& \longrightarrow & \mathbb{R} \\
&x& \longmapsto & \log\left(1+\frac{u(x)}{\delta}\right)
\end{array}
$$
for all $\delta>0$.
We have $F_{\delta}(y)=0$ for all $y\in Z$. Therefore, if $x \in B_{r}(x_{0})$ and $y \in Z$ 
$$
|F_{\delta}(x)|^{2} = \frac{|F_{\delta}(x)-F_{\delta}(y)|^{2}}{|x-y|^{n+2s}}|x-y|^{n+2s}
$$
Integrating with respect to $ y \in Z $ we get
$$
\begin{array}{ll}
|Z||F_{\delta}(x)|^{2} &= \int_{Z}\frac{|F_{\delta}(x)-F_{\delta}(y)|^{2}}{|x-y|^{n+2s}}|x-y|^{n+2s}dy \\
&\leq 2r^{n+2s}\int_{Z}\frac{|F_{\delta}(x)-F_{\delta}(y)|^{2}}{|x-y|^{n+2s}}dy
\end{array}
$$
Now, integrating with respect to $x \in B_{r}$ we get
$$
\begin{array}{ll}
\int_{B_{r}(x_{0})}|F_{\delta}(x)|^{2}dx & \leq \frac{1}{|Z|} 2r^{n+2s}\int_{B_{r}(x_{0})}\int_{Z}\frac{|F_{\delta}(x)-F_{\delta}(y)|^{2}}{|x-y|^{n+2s}}dydx
\\ &\leq \frac{1}{|Z|} 2r^{n+2s}\int_{B_{r}(x_{0})}\int_{B_{r}(x_{0})}\frac{|F_{\delta}(x)-F_{\delta}(y)|^{2}}{|x-y|^{n+2s}}dydx \\
& = \frac{1}{|Z|} 2r^{n+2s}\int_{B_{r}(x_{0})}\int_{B_{r}(x_{0})}\left|\log \left(\frac{\delta+u(x)}{\delta+u(y)}\right)\right|^{2}\frac{1}{|x-y|^{n+2s}}dxdy \\
& \leq \frac{1}{|Z|} 2r^{n+2s}\left(Cr^{n-2s}+\int_{B_{2r}}a(x)dx\right) \\
& =  \frac{1}{|Z|} 2Cr^{2n}+ \frac{1}{|Z|} 2r^{n+2s}\int_{B_{2r}}a(x)dx:=L.
\end{array}
$$
The number $L$ does not depend on $\delta$. In short, we have proved that
$$
\int_{B_{r}(x_{0})}\left|\log\left(1+\frac{u(x)}{\delta}\right)\right|^{2}dx \leq C
$$
for some constant $C>0$ and $C$ does not depend on $\delta$. If $u(x) \neq 0$ then $F_{\delta}(x) \rightarrow \infty$ when $\delta \rightarrow 0$. By Fatou's lemma, if $|B_{r} \cap Z^{c}|>0$,
$$
+\infty \leq \liminf_{\delta \rightarrow 0}\int_{B_{r}\cap Z^{c}}|F_{\delta}(x)|^{2} \leq C.
$$
Therefore $|Z|=|B_{r}|$ and $u=0$ almost everywere in $B_{r}(x_{0})$. Now, we define
$$
A=\left\{B_{r}(x); r>0, x \in \mathbb{R}^{n}, u>0\ in\ B_{r}(x)\right\}
$$
$$
B=\left\{B_{r}(x); r>0, x \in \mathbb{R}^{n}, u=0\  in\ B_{r}(x)\right\}
$$
$$
S=\bigcup_{V\in A}V
$$
and
$$
W=\bigcup_{V\in B}V
$$
$S$ and $W$ are open sets. Consider $x \in \mathbb{R}^{n}$ and $r>0$. We have two possibilities, either $u \neq 0$ in $B_{r}(x)$ or $u=0$ in $B_{r}(x)$. If $u \neq 0$ in $B_{r}$ then $u>0$ in $B_{r}$. In this case, $x \in S$. If $u=0$ in $B_{r}(x)$, then $x \in W$. Consequently
$$
\mathbb{R}^{n}= S \cup W.
$$
By connectedness, we should have $S=\emptyset$ or $W=\emptyset$. If $\mathbb{R}^{n}=S$ then $u>0$ almost everywere in $\mathbb{R}^{n}$. If $\mathbb{R}^{n}=W$ then $u=0$ almost everywere in $\mathbb{R}^{n}$.
\end{proof}

\begin{corollary}\label{cl45}
The solution $u$ found in Theorem \ref{th31} is positive in the following sense, $ u>0$ almost everywere in $\mathbb{R}^{3}$.
\end{corollary}

\begin{proof}
For some $v\in H^{s}(\mathbb{R}^{3})$, with $v \geq0$ almost everywere, we have 
$$
\begin{array}{ll}
\frac{(\zeta)}{2}\int_{\mathbb{R}^{3}}\int_{\mathbb{R}^{3}}\frac{(u(x)-u(y))(v(x)-v(y))}{|x-y|^{3+2s}}dxdy&+ \int_{\mathbb{R}^{3}}V(x)uvdx \\&+\int_{\mathbb{R}^{3}}\phi_{u}uvdx=\int_{\mathbb{R}^{3}}g(u)vdx \geq 0.
\end{array}
$$
If we define $a(x)=\frac{2}{\zeta}(V(x)+\phi_{u}(x))$, we have that $a \in L^{1}_{loc}(\mathbb{R}^{3})$, because $L^{2^{\ast}_{t}}(\mathbb{R}^{3})\subset L^{1}_{loc}(\mathbb{R}^{3})$ and $V$ is continuous. By $(V_{0})$ and Lemma \ref{lm1} we have $a(x)>0$ in $\mathbb{R}^{3}$. 
Thereby, 
$$
\int_{\mathbb{R}^{3}}\int_{\mathbb{R}^{3}}\frac{(u(x)-u(y))(v(x)-v(y))}{|x-y|^{3+2s}}dxdy + \int_{\mathbb{R}^{3}}a(x)uvdx \geq 0.
$$
for all $v\in H^{s}(\mathbb{R}^{3})$ with $v\geq0$.
But $u \neq 0$. Then, Theorem \ref{th44} implies that $u>0$ almost everywere in $\mathbb{R}^{3}$.
\end{proof}

\begin{remark}\label{rm46}
Define $\mathcal{N}=\left\{u \in H^{s}(\mathbb{R}^{3})\setminus\left\{0\right\};I'(u)u=0\right\}$, where
$$
\begin{array}{ll}

\begin{array}{ll}I(u)=\frac{\zeta(s)}{4}\int_{\mathbb{R}^{3}}\int_{\mathbb{R}^{3}}&\frac{(u(x)-u(y))^{2}}{|x-y|^{3+2s}}dxdy+\frac{1}{2}\int_{\mathbb{R}^{3}}V(x)u^{2}dx\\
&+\frac{1}{4}\int_{\mathbb{R}^{3}}\phi_{u}u^{2}dx-\int_{\mathbb{R}^{3}}F(u)dx.
\end{array}
\end{array}
$$
If $f$ satisfies $(f_{1})-(f_{5})$ then
$$
I_{\infty}=inf_{u\in \mathcal{N}} I(u)
$$
coincides with the mountain pass level associated with $I$.
\end{remark}

\begin{theorem}
If $f$ satisfies $(f_{1})-(f_{5})$ and $V$ satisfies $(V_{0})$ and $(V_{1})$, then the problem $(P)$ has a ground state solution.
\end{theorem}

\begin{proof}
Taking the following Euler-Lagrange functional
$$
\begin{array}{cccl}
I:&H^{s}(\mathbb{R}^{3})&\longrightarrow&\mathbb{R}\\
&u& \longmapsto &
\begin{array}{ll}\frac{\zeta(s)}{4}\int_{\mathbb{R}^{3}}\int_{\mathbb{R}^{3}}&\frac{(u(x)-u(y))^{2}}{|x-y|^{3+2s}}dxdy+\frac{1}{2}\int_{\mathbb{R}^{3}}V(x)u^{2}dx\\
&+\frac{1}{4}\int_{\mathbb{R}^{3}}\phi_{u}u^{2}dx-\int_{\mathbb{R}^{3}}F(u)dx,
\end{array}
\end{array}
$$
and following with the same ideas of Theorem \ref{th31}, we prove that there is a nonzero solution $u$ to the system $(P)$. Also, we prove that there is a Cerami's sequence $\left\{w_{n}\right\}$ in the montain pass level associated with $I$ converging to $u$. By Remark \ref{rm1} and Fatou's lemma
$$
\begin{array}{ll}
4c&= \liminf_{n \rightarrow \infty}\left(4I(w_{n})-I'(w_{n})w_{n}\right) \\
& = \liminf_{n \rightarrow \infty}\left(||w_{n}||^{2}+ \int_{\mathbb{R}^{3}}H(w_{n})dx\right) \\
& \geq  \liminf_{n \rightarrow \infty}||w_{n}||^{2}+  \liminf_{n \rightarrow \infty}\int_{\mathbb{R}^{3}}H(w_{n})dx \\
& \geq ||u||^{2}+\int_{\mathbb{R}^{3}}H(u)dx \\
& = 4I(u)-I'(u)u \\
& = 4I(u).
\end{array}
$$
where $H(u)=uf(u)-4F(u)$.
By definition $u \in \mathcal{N}$. Then $I(u) \leq \inf_{u\in \mathcal{N}}I(u)$. By Remark \ref{rm46}
$$
I(u)=\inf_{u\in \mathcal{N}}I(u).
$$
\end{proof}

\section{Asymptotically Periodic Potential}
In this section, we study the problem $(P)$, when we consider $V$ satisfying the condition $(V_{0})$ and
\begin{itemize}
\item[$(V_3)$\ ] There is a function $V_{p}$ satisfying $(V_{1})$ such that
$$
\lim_{|x|\rightarrow \infty}|V(x)-V_{p}(x)|=0;
$$
\item[$(V_4)$\ ] $V(x)\leq V_{p}(x)$ and there is a open set $\Omega \subset \mathbb{R}^{3}$ with $|\Omega|>0$ and $V(x)< V_{p}(x)$ in $\Omega$.
\end{itemize}
Here $V_{p}$ is a periodic continuous potential. This case follows the same ideas already studied in Schr\"odinger-Poisson system with asymptotically periodic potential in \cite{amsss}. We are writing this case to make a most complete work for the reader.
\begin{theorem}
Suppose that $V$ satisfies $(V_{0})$, $(V_{3})$, $(V_{4})$ and $f$ satisfies $(f_{1})-(f_{5})$. Then, the problem $(P)$ has a ground state solution.
\end{theorem}

\begin{proof}
We can  define in  $H^{s}(\mathbb{R}^{3})$ the norm,
$$
||u||_{p}=\left(\frac{\zeta}{2}\int_{\mathbb{R}^{3}}\int_{\mathbb{R}^{3}}\frac{(u(x)-u(y))^{2}}{|x-y|^{3+2s}}dxdy+\int_{\mathbb{R}^{3}}V_{p}(x)u^{2}dx\right)^{\frac{1}{2}}.
$$
Consider the functional $I_{p}$ 
$$
I_{p}(u)=\frac{1}{2}||u||_{p}^{2}+\frac{1}{4}\int_{\mathbb{R}^{3}}\phi_{u}u^{2}dx - \int_{\mathbb{R}^{3}}F(u)dx.
$$
We claim that there is $w_{p} \in H^{s}(\mathbb{R}^{3})$ such that $I_{p}'(w_{p})=0$ and $I_{p}(w_{p})=c_{p}$, where $c_{p}$ is the mountain pass level associated with $I_{p}$. 
We will consider another norm in  $H^{s}(\mathbb{R}^{3})$.
$$
||u||=\left(\frac{\zeta}{2}\int_{\mathbb{R}^{3}}\int_{\mathbb{R}^{3}}\frac{(u(x)-u(y))^{2}}{|x-y|^{3+2s}}dxdy+\int_{\mathbb{R}^{3}}V(x)u^{2}dx\right)^{\frac{1}{2}}.
$$
Then, we define 
$$
I(u)=\frac{1}{2}||u||^{2}+\frac{1}{4}\int_{\mathbb{R}^{3}}\phi_{u}u^{2}dx - \int_{\mathbb{R}^{3}}F(u)dx.
$$
The functional $I$  has a mountain pass geometry. If $c$ is the mountain pass level associated with $I$ then $c<c_{p}$. Indeed, there is a $t_{\ast}$ such that $t_{\ast}w_{p} \in \mathcal{N}$ (see remark \ref{rm46}) and it is the unique with this property. Then 
$$
\begin{array}{ll}
c &\leq I(t^{\ast}w_{p})\\
&<I_{p}(t^{\ast}w_{p}) \\
& \leq \max_{t\geq 0}I_{p}(tw_{p}) \\
& = I_{p}(w_{p})=c_{p}
\end{array}
$$ Consider $\{u_{n}\}_{n \in \mathbb{N}}$ a Cerami's sequence at the mountain pass level $c$ associated with $I$. Similarly to the periodic case, we prove that the sequence $\{u_{n}\}$ is bounded and therefore, converges weakly to $u \in H^{s}(\mathbb{R}^{3})$. Additionally $I'(u)=0$. Now we will prove that $u \neq 0$. Suppose that $u=0$. 
Regarding the sequence $\left\{u_{n}\right\}$, the following equalities are true
$\newline$
\begin{enumerate}
\item $\lim\limits_{n \rightarrow \infty}\int_{\mathbb{R}^{3}}|V(x)-V_{p}(x)|u_{n}^{2}dx= 0$
\item $\lim\limits_{n \rightarrow \infty}|||u_{n}||-||u_{n}||_{p}|=0$.
\item $\lim\limits_{n \rightarrow \infty}|I_{p}(u_{n})-I(u_{n})|=0$
\item $\lim\limits_{n \rightarrow \infty}|I_{p}'(u_{n})u_{n}-I'(u_{n})u_{n}|=0$
\end{enumerate}
We will prove (1). The limits (2), (3) and (4) are immediate consequences of (1). Consider $\epsilon>0$ and $A>0$ such that $||u_{n}||_{2}^{2}< A$ for all $n \in \mathbb{N}$.  By $(V_{3})$, there is $R>0$ such that, for all $|x|>R$ we have
$$
|V(x)-V_{p}(x)|< \frac{\epsilon}{2A}.
$$
But $\{u_{n}\}$ converges weakly to $u=0$. Then $u_{n}\rightarrow 0$ in $L^{2}(B_{R}(0))$. This convergence implies that there is $n_{0} \in \mathbb{N}$ such that
$$
\int_{B_{R}(0)}|V(x)-V_{p}(x)|u_{n}^{2}dx< \frac{\epsilon}{2}
$$ 
for all $n \geq n_{0}$. Then, if $n\geq n_{0}$
$$
\begin{array}{ll}
&\int_{\mathbb{R}^{3}}|V(x)-V_{p}(x)|u_{n}^{2}dx\\
&=\int_{B_{R}(0)}|V(x)-V_{p}(x)|u_{n}^{2}dx+\int_{(B_{R}(0))^{c}}|V(x)-V_{p}(x)|u_{n}^{2}dx \\
&<\frac{\epsilon}{2}+\frac{\epsilon}{2}=\epsilon.
\end{array}
$$
\newline
Consider $s_{n}>0$ such that 
$$
s_{n}u_{n} \in \mathcal{N}_{p}
$$
for every $n \in \mathbb{N}$.
Where $N_{p}=\left\{u \in H^{s}(\mathbb{R}^{3})\setminus\left\{0\right\};I_{p}'(u)u=0\right\}$. We claim that $\limsup_{n \rightarrow \infty}s_{n}\leq 1$. Indeed, otherwise, there is $\delta>0$ such that, passing to a subsequence if necessary, we can assume that $s_{n}\geq 1+\delta$ for all $n \in \mathbb{N}$. By $(4)$ we have $I_{p}'(u_{n})u_{n}\rightarrow 0$, that is,
$$
\begin{array}{ll}
||u_{n}||_{p}^{2}+\int_{\mathbb{R}^{3}}\phi_{u_{n}}u_{n}^{2}dx = \int_{\mathbb{R}^{3}}f(u_{n})u_{n}dx+o_{n}(1)
\end{array}
$$
From $s_{n}u_{n}\in \mathcal{N}_{p}$ we have $I_{p}'(s_{n}u_{n})u_{n}=0.$ Equivalently
$$
s_{n}||u_{n}||_{p}^{2}+s_{n}^{3}\int_{\mathbb{R}^{3}}\phi_{u_{n}}u_{n}^{2}dx = \int_{\mathbb{R}^{3}}f(s_{n}u_{n})u_{n}dx
$$
Therefore
\begin{equation}\label{eq53}
\int_{\mathbb{R}^{3}}\left[\frac{f(s_{n}u_{n})}{(s_{n}u_{n})^{3}}-\frac{f(u_{n})}{(u_{n})^{3}}\right]u_{n}^{4}dx = \left(\frac{1}{s_{n}^{2}}-1\right)||u_{n}||_{p}^{2}+o_{n}(1) \leq o_{n}(1).
\end{equation}
If $\{u_{n}\}_{n \in \mathbb{N}}$ converges to $0$ in  $L^{q}(\mathbb{R}^{3})$ for all $q \in (2,2^{\ast}_{s})$, then by Lemma \ref{lm1}
$$
||u_{n}||^{2}\leq||u_{n}||^{2}+\int_{\mathbb{R}^{3}}\phi_{u_{n}}u_{n}^{2}dx = \int_{\mathbb{R}^{3}}f(u_{n})u_{n}+I'(u_{n})u_{n}
$$
consequently $\{u_{n}\}$ would have limit $0$ in $H^{s}(\mathbb{R}^{3})$ and this would contradict the fact that $c>0$. Therefore, there is a sequence $\{y_{n}\} \subset \mathbb{Z}^{n}$, $R>0$ and $\beta>0$ such that
$$
\int_{B_{R}(y_{n})}u_{n}^{2}dx \geq \beta>0
$$
Taking $v_{n}(x):=u_{n}(x+y_{n})$ we have $||v_{n}||=||u_{n}||$ and therefore we can assume that $\{v_{n}\}_{n \in \mathbb{N}}$ converges weakly to some $v \in H^{s}(\mathbb{R}^{3})$. Note that
$$
\int_{B_{R}(0)}v^{2}dx \geq \beta>0
$$
The inequality $(\ref{eq53})$, Remark \ref{rm1} and Fatou's lemma imply that
$$
\begin{array}{ll}
0<&\int_{\mathbb{R}^{3}}\left[\frac{f((1+\delta)v)}{((1+\delta)v)^{3}}-\frac{f(v)}{(v)^{3}}\right]v^{4}dx \\
& \leq \liminf_{n \rightarrow \infty}\int_{\mathbb{R}^{3}}\left[\frac{f((1+\delta)v_{n})}{((1+\delta)v_{n})^{3}}-\frac{f(v_{n})}{(v_{n})^{3}}\right]v_{n}^{4}dx \\
& \leq \liminf_{n \rightarrow \infty}\int_{\mathbb{R}^{3}}\left[\frac{f(s_{n}v_{n})}{(s_{n}v_{n})^{3}}-\frac{f(v_{n})}{(v_{n})^{3}}\right]v_{n}^{4}dx \\
& =  \liminf_{n \rightarrow \infty}\int_{\mathbb{R}^{3}}\left[\frac{f(s_{n}u_{n})}{(s_{n}u_{n})^{3}}-\frac{f(u_{n})}{(u_{n})^{3}}\right]u_{n}^{4}dx \\
& \leq \liminf_{n \rightarrow \infty}o_{n}(1)=0.
\end{array}
$$
The last inequality is a contradiction. Therefore $\limsup_{n \rightarrow \infty}s_{n}\leq 1$. Now, we will prove that for $n$ large enough, $s_{n}> 1$. Suppose that the statement is false. In this case, passing to a subsequence if necessary, we can assume that $s_{n}\leq 1$ for all $n \in \mathbb{N}$. Note that by $(f_{5})$, the function $H(u):=uf(u)-4F(u)$ is increasing in $|u| \neq 0$. Then
$$
\begin{array}{ll}
4c_{p}&=4\inf_{u \in N_{p}}I_{p}(u) \\
& \leq 4I_{p}(s_{n}u_{n}) \\
& = 4I_{p}(s_{n}u_{n}) - I_{p}'(s_{n}u_{n})(s_{n}u_{n})\\
& = s_{n}^{2}||u_{n}||_{p}^{2}+\int_{\mathbb{R}^{3}}f(s_{n}u_{n})(s_{n}u_{n})-4F(s_{n}u_{n})dx \\
& \leq ||u_{n}||_{p}^{2}+\int_{\mathbb{R}^{3}}f(u_{n})(u_{n})-4F(u_{n})dx \\
& \leq 4I(u_{n})-I'(u_{n})u_{n}+\int_{\mathbb{R}^{3}}|V(x)-V_{p}(x)|u_{n}^{2}dx.
\end{array}
$$
This implies that
$$
4c_{p} \leq 4c.
$$
But, this last inequality is false, because we have proved that $c<c_{p}$. Therefore, we have that $s_{n}> 1$ for $n$ large enough. Then, about the sequence $\left\{s_{n}\right\}$ we have proved that
$$
1 \leq \liminf_{n \rightarrow \infty}s_{n} \leq \limsup_{n \rightarrow \infty}s_{n} \leq 1.
$$
and therefore
\begin{equation}\label{eq54}
\lim\limits_{n \rightarrow \infty}s_{n}=1.
\end{equation}
The Fundamental Theorem of Calculus implies that
\begin{equation}\label{eq55}
\begin{array}{ll}
\int_{\mathbb{R}^{3}}F(s_{n}u_{n})dx-\int_{\mathbb{R}^{3}}F(u_{n})dx = \int_{1}^{s_{n}}\left[\int_{\mathbb{R}^{3}}f(\tau u_{n})u_{n}dx \right] d \tau.
\end{array}
\end{equation}
Also, by $(f_{3})$ we obtain $C>0$ such that
\begin{equation}\label{eq56}
\int_{\mathbb{R}^{3}}f(\tau u_{n})u_{n}dx \leq C(s_{n}||u_{n}||^{2}+s_{n}^{p-1}||u_{n}||^{p}).
\end{equation}
for all $\tau \in (1,s_{n})$.
We have that the sequence $\{u_{n}\}$ is bounded. Then, by (\ref{eq54}), (\ref{eq55}) and (\ref{eq56})
$$
\int_{\mathbb{R}^{3}}F(s_{n}u_{n})dx-\int_{\mathbb{R}^{3}}F(u_{n})dx = o_{n}(1).
$$
Then
$$
\begin{array}{ll}
&I_{p}(s_{n}u_{n})-I_{p}(u_{n})\\
& = \frac{(s_{n}^{2}-1)}{2}||u_{n}||^{2}+\frac{(s_{n}^{4}-1)}{4}\int_{\mathbb{R}^{3}}\phi_{u_{n}}u_{n}^{2}dx - \int_{\mathbb{R}^{3}}F(s_{n}u_{n})dx+\int_{\mathbb{R}^{3}}F(u_{n})dx \\
& = o_{n}(1)
\end{array}
$$
because $\{u_{n}\}$ is bounded and 
$\int_{\mathbb{R}^{3}}\phi_{u_{n}}u_{n}^{2}dx = ||\phi_{u_{n}}||_{\dot{H}^{t}(\mathbb{R}^{3})}^{2} \leq C||u_{n}||^{4}$.
By $(3)$ 
$$
\begin{array}{ll}
c_{p}&\leq I_{p}(s_{n}u_{n}) \\
& = I_{p}(u_{n})+o_{n}(1).\\
&= I(u_{n})+o_{n}(1)
\end{array}
$$
Taking $n \rightarrow \infty$ we obtain
$$
c_{p} \leq c
$$
But, this last inequality is false, because we have proved that $c<c_{p}$. This contradiction was generated because we assumed that $u = 0$. It follows that $u$ is nontrivial. 
Particularly
$$
I(u) \geq \inf_{u \in \mathcal{N}}I(u).
$$
As in the periodic case
$$
I(u)\leq c=\inf_{u \in \mathcal{N}}I(u).
$$
Therefore $u$ is a ground state solution for the system $(P)$.
\end{proof}

\end{document}